\documentclass{amsart}

\newtheorem{thm}{Theorem}[section]
\newtheorem{con}[thm]{Conjecture}
\newtheorem{cor}[thm]{Corollary}
\newtheorem{lem}[thm]{Lemma}
\newtheorem{prop}[thm]{Proposition}
\newcommand{\setf}[2]{\left\{ \left.{#1} \right| {#2} \right\}}

\theoremstyle{definition}

\newtheorem{defns}[thm]{Definitions}

\newtheorem{rem}[thm]{Remark}

\DeclareMathOperator{\hc}{hc}
\DeclareMathOperator{\lc}{lc}
\DeclareMathOperator{\hdeg}{hdeg}
\DeclareMathOperator{\ldeg}{ldeg}
\DeclareMathOperator{\Um}{Um}
\DeclareMathOperator{\E}{E}
\DeclareMathOperator{\GL}{GL}
\DeclareMathOperator{\G}{G}
\DeclareMathOperator{\SL}{SL}
\begin{document}

\title[stably free modules over Laurent polynomial rings]{\textbf{On stably free modules over Laurent polynomial rings}}
\author[A.Abedelfatah]{Abed Abedelfatah}
\address{Department of Mathematics, University of Haifa, Mount Carmel, Haifa 31905, Israel}
\email{abedelfatah@gmail.com}
\keywords{Stably free modules, Hermite rings, Unimodular rows, Laurent polynomial rings, Constructive Mathematics}
\begin{abstract}
We prove constructively that for any finite-dimensional commutative ring $R$ and $n\geq\dim (R)+2$, the group $\E_{n}(R[X,X^{-1}])$ acts transitively on $\Um_{n}(R[X,X^{-1}])$. In particular, we obtain that for any finite-dimensional ring $R$, every finitely generated stably free module over $R[X,X^{-1}]$ of rank $>\dim R$ is free, i.e., $R[X,X^{-1}]$ is $(\dim R)$-Hermite.
\end{abstract}
\maketitle
\section{Introduction}

We denote by $R$ a commutative ring with unity and  $\mathbb{N}$ the set of non-negative integers. $\Um_{n}(R)$ is the set of unimodular rows of length $n$ over $R$, that is all $(x_{0},\dots,x_{n-1})\in R^{n}$ such that $x_{0}R+\dots+x_{n-1}R=R$. If $u,v\in \Um_{n}(R)$ and $G$ is a subgroup of $\GL_{n}(R)$, we write $u\sim_{\G}v$ if there exists $g$  in $G$ such that $v=ug$. Recall that $\E_{n}(R)$ denotes the subgroup of $\GL_{n}(R)$, generated by all $\E_{ij}(a):=I_{n}+ae_{ij}$ (where $i\neq j$, $a\in R$ and $e_{ij}$ denotes the $n\times n$- matrix whose only non-zero entry is 1 on the $(i,j)$- th place). We abbreviate the notation $u\sim_{\E_{n}(R)}v$ to $u\sim_{\E}v$. We say that a ring $R$ is Hermite (resp. d-Hermite ) if any finitely generated stably free $R$-module ( resp., any finitely generated stably free $R$-module of rank $>d$ ) is free.

In \cite{F}, A.A.Suslin proved:
\begin{thm}\emph{(A.A.Suslin)}\\
If $R$ is a Noetherian ring and $$A=R[X_{1}^{\pm1},\dots,X_{k}^{\pm1},X_{k+1},\dots,X_{n}].$$
Then for $n\geq\max{(3,\dim(R)+2)}$ the group $\E_{n}(A)$ acts transitively on $\Um_{n}(A)$.
\end{thm}
In particular, we obtain that $\E_{n}(R[X,X^{-1}])$ acts transitively on $\Um_{n}(R[X,X^{-1}])$ for any Noetherian ring $R$, where $n\geq\max{(3,\dim(R)+2)}$.
In \cite{G}, I.Yengui proved:
\begin{thm}\emph{(I.Yengui)}\\
Let $R$ be a ring of dimension $d$, $n\geq d+1$, and let $f\in \Um_{n+1}(R[X])$. Then there exists $E\in \E_{n+1}(R[X])$ such that $f\cdot E=e_{1}$.
\end{thm}
In this article we generalize by proving:
\begin{thm}\label{30}
For any finite-dimensional ring $R$, $\E_{n}(R[X,X^{-1}])$ acts transitively on $\Um_{n}(R[X,X^{-1}])$, where $n\geq \dim(R)+2$.
\end{thm}

This gives a positive answer to Yengui's question (Question $9$ of \cite{G}). The proof we give is a close adaptation of Yengui's proof to the Laurent case.

\section{Preliminary results on unimodular rows}

A.A.Suslin proved in \cite{F}, that if $f=(f_{0},\dots,f_{n})\in \Um_{n+1}(R[X])$, where $f_{1}$ is unitary and $n\geq 1$, then there exists $w\in \SL_{2}(R[X])\cdot\E_{n+1}(R[X])$ such that $f\cdot w=e_{1}$. In fact, this theorem is a crucial point in his proof of Serre's conjecture. R.A.Rao generalized in [\cite{D}, Corollary 2.5] by proving:

\begin{thm}\emph{(R.A.Rao, \cite{D})}\label{70}\\
Let $f=(f_{0},\dots,f_{n})\in \Um_{n+1}(R[X])$, where $n\geq 2$. If some $f_{i}$ is unitary, then $f$ is completable to a matrix in $\E_{n}(R[X])$.
\end{thm}

Recall that the boundary ideal of an element $a$ of a ring $R$ is the ideal $\mathcal{I}(a)$ of $R$ generated by $a$ and all $y\in R$ such that $ay$ is nilpotent. Moreover, $\dim R\leq d\Leftrightarrow\dim(R/\mathcal{I}(a))\leq d-1$ for all $a\in R$ \cite{C}.

\begin{thm}\emph{[\cite{B}, Theorem 2.4]}\label{73}\\
Let $R$ be a ring of dimension $\leq d$ and $a=(a_{0},\dots,a_{n})\in \Um_{n+1}(R)$ where $n\geq d+1$, then there exist $b_{1},\dots,b_{n}\in R$ such that $$\langle a_{1}+b_{1}a_{0},\dots,a_{n}+b_{n}a_{0}\rangle=R$$
\end{thm}

In fact, we can obtain a stronger result if $f\in \Um_{n+1}(R_{S})$, where $S$ is a multiplicative subset of $R$:

\begin{prop}\label{41}
Let S be a multiplicative subset of $R$ such that
$S^{-1}R$ has dimension $d$. Let $(a_{0},\dots,a_{n})\in M_{n+1}(R)$ be a row such that $(\frac{a_{0}}{1},\dots,\frac{a_{n}}{1})\in \Um_{n+1}(S^{-1}R)$,
where $n> d$. Then there exist $b_{1},\dots,b_{n}\in R$ and $s\in S$ such that $$s\in (a_{1}+b_{1}a_{0})R+\dots+(a_{n}+b_{n}a_{0})R.$$
\end{prop}
\begin{proof}
By induction on $d$, if $d=0$ then $R_{S}/\mathcal{I}(\frac{a_{n}}{1})\cong (R/J)_{\overline{S}}$  is trivial, where $\overline{S}=\setf{s+J}{s\in S}$, $J=i^{-1}(\mathcal{I}(\frac{a_{n}}{1}))$, and $i:R\rightarrow R_{S}$ is the natural homomorphism. So $1\in\langle \frac{a_{n}}{1},\frac{b_{n}}{1}\rangle$ in $R_{S}$, where $b_{n}\in R$ and $\frac{a_{n}b_{n}}{1}$ is nilpotent. Since $1\in \langle\frac{a_{1}}{1},\dots,\frac{a_{n-1}}{1},\frac{a_{n}}{1},\frac{b_{n}a_{0}}{1}\rangle$, so by [\cite{B}, Lemma 2.3], $1\in \langle\frac{a_{1}}{1},\dots,\frac{a_{n-1}}{1},\frac{a_{n}+b_{n}a_{0}}{1}\rangle$, i.e., there exist $s\in S$ such that $s\in a_{1}R+\dots+a_{n-1}R+(a_{n}+b_{n}a_{0})R$.

Assume now $d>0$. By the induction assumption with respect to the ring $R_{S}/\mathcal{I}(\frac{a_{n}}{1})\cong (R/J)_{\overline{S}}$ we can find $\bar{b}_{1},\dots,\bar{b}_{n-1}\in R/J$ such that $$\langle\frac{\bar{a}_{1}+\bar{b}_{1}\bar{a}_{0}}{\overline{1}},\dots,\frac{\bar{a}_{n-1}+\bar{b}_{n-1}\bar{a}_{0}}{\overline{1}}\rangle=(R/J)_{\overline{S}}.$$  So $\langle\overline{\frac{a_{1}+b_{1}a_{0}}{1}},\dots,\overline{\frac{a_{n-1}+b_{n-1}a_{0}}{1}}\rangle=R_{S}/\mathcal{I}(\frac{a_{n}}{1})$, this means that $$\langle\frac{a_{1}+b_{1}a_{0}}{1},\dots,\frac{a_{n-1}+b_{n-1}a_{0}}{1},\frac{a_{n}}{1},\frac{b_{n}}{1}\rangle=R_{S}$$  where $\frac{a_{n}b_{n}}{1}$ is nilpotent. So by [\cite{B}, Lemma 2.3] $$\langle\frac{a_{1}+b_{1}a_{0}}{1},\dots,\frac{a_{n-1}+b_{n-1}a_{0}}{1},\frac{a_{n}+b_{n}a_{0}}{1}\rangle=R_{S}.$$
\end{proof}

Let $f\in \Um_{n+1}(R[X])$, where $n\geq \frac{d}{2}+1$, with $R$ a local ring of dimension $d$. M.Roitman's argument in [\cite{E}, Theorem 5], shows how one could decrease the degree of all but one (special) co-ordinate of $f$. In the absence of a monic polynomial as a co-ordinate of $f$ he uses a Euclid's algorithm and this is achieved via,

\begin{lem}\emph{(M.Roitman, [\cite{E}, Lemma 1])}\label{2}\\
Let $(x_{0},\dots,x_{n})\in \Um_{n+1}(R)$, $n\geq2$, and let $t$ be an element of $R$ which is invertible $\bmod{(Rx_{0}+\dots+Rx_{n-2})}$. Then $$(x_{0},\dots,x_{n})\sim_{\E_{n+1}(R)}(x_{0},\dots,t^{2}x_{n})\sim_{\E_{n+1}(R)}(x_{0},\dots ,tx_{n-1},tx_{n}).$$
\end{lem}

\section{The main results}

\begin{defns}
Let $f\in R[X,X^{-1}]$ be a nonzero Laurent polynomial. We denote $\deg(f)=\hdeg(f)-\ldeg(f)$, where $\hdeg(f)$ and $\ldeg(f)$ denote respectively the highest and the lowest degree of $f$.

Let $\hc(f)$ and $\lc(f)$ denote respectively the coefficients of the highest and the lowest degree term of $f$.
An element $f\in R[X,X^{-1}]$ is called a doubly unitary if $\hc(f),\lc(f)\in U(R)$.
\end{defns}
For example, $\deg(X^{-3}+X^{2})=5$.
\begin{lem}\label{40}
Let $f_{1},\dots,f_{n}\in R[X,X^{-1}]$ such that  $\hdeg(f_{i})\leq k-1$, $\ldeg(f_{i})\geq -m$ for all $1\leq i\leq n$. Let $f\in R[X,X^{-1}]$ with $\hdeg(f)=k,~ \ldeg(f)\geq -m$, where $k,m\in \mathbb{N}$. Assume that $\hc(f)\in U(R)$ and the coefficients of $f_{1},\dots,f_{n}$ generate the ideal $(1)$ of $R$, then $I=\langle f_{1},\dots,f_{n},f\rangle$ contains a polynomial $h$ of $\hdeg(h)=k-1$, $\ldeg(h)\geq -m$ and $\hc(h)\in U(R)$.
\end{lem}
\begin{proof}
Since $X^{m}f_{1},\dots,X^{m}f_{n},X^{m}f\in R[X]$, by [\cite{A}, \S4, Lemma 1(b)], $I$ contains a polynomial $h_{1}\in R[X]$ of degree $m+k-1$ which is unitary. So $h=X^{-m}h_{1}\in I$ of $\hdeg(h)=k-1$, $\ldeg(h)\geq -m$ and $\hc(h)\in U(R)$.
\end{proof}

\begin{prop}\label{42}
Let $I\unlhd R[X,X^{-1}]$ be an ideal, $J\unlhd R$, such that $I$ contains a doubly unitary polynomial. If $I+J[X,X^{-1}]=R[X,X^{-1}]$ then $(I\cap R)+J=R$.
\end{prop}
\begin{proof}
Let us denote by $h_1$ a doubly unitary polynomial in $I$. Since $I+J[X,X^{-1}]=R[X,X^{-1}]$, there exist $h_2\in I$ and $h_3\in J[X,X^{-1}]$ such that $h_2+h_3=1$. Let $g_i=X^{-\ldeg(h_i)}h_i$, for $i=1,2,3$. Since $X^l\in \sum_{i=1}^{3}g_{i}R[X]$, for some $l\geq 0$, and $g_{1}\equiv u\bmod{XR[X]}$, where $u\in U(R)$, we obtain that $\langle g_1,g_2,g_3\rangle=\langle1\rangle$ in $R[X]$. By [\cite{H}, Lemma 2], we obtain $(\langle g_1,g_2\rangle\cap R)+J=R$. So $(I\cap R)+J=R$.
\end{proof}

\begin{thm}\label{43}
Let $f=(f_{0},\dots,f_{n})\in \Um_{n+1}(R[X,X^{-1}])$, where $n\geq 2$. Assume that $f_{0}$ is a doubly unitary polynomial, then $$f\sim_{\E_{n+1}(R[X,X^{-1}])}(1,0,\dots,0).$$
\end{thm}
\begin{proof}
By (\ref{2}), $f\sim_{\E}(X^{-\ldeg(f_{0})}f_{0},X^{-\ldeg(f_{0})}f_{1},f_{2},\dots,f_{n})\sim_{\E}\\(X^{-\ldeg(f_{0})}f_{0},X^{-\ldeg(f_{0})+2k}f_{1},X^{2k}f_{2},\dots,X^{2k}f_{n})=(g_{0},\dots,g_{n})$
where $k\in \mathbb{N}$. For sufficiently big $k$, we obtain that $g_{0},\dots,g_{n}\in R[X]$. Clearly, $X^{l}\in \sum_{i=0}^{n}g_{i}R[X]$ for some $l\geq 0$. But $g_{0}\equiv u\bmod{XR[X]}$, where $u\in U(R)$, then $X^{l}R[X]+g_{0}R[X]=R[X]$, so $g\in \Um_{n}(R[X])$. By (\ref{70}), $g\sim_{\E} e_{1}$.
\end{proof}

\begin{rem}\label{44}
Let $a=(a_{1},\dots,a_{n})\in \Um_{n+1}(R)$, where $n\geq 2$. If $$a\sim_{\E_{n}(R/\mathrm{Nil}(R))}e_{1}$$ then $a\sim_{\E_{n}(R)}e_{1}$.
\end{rem}

\begin{prop}\label{48}
If $R$ is a zero-dimensional ring and $f=(f_{0},\dots,f_{n})\\\in \Um_{n+1}(R[X,X^{-1}])$, where $n\geq 1$. Then $$f\sim_{\E}e_{1}.$$
\end{prop}
\begin{proof}
We prove by induction on $\deg f_0+\deg f_1$. We may assume that $R$ is reduced ring. Let $a=\hc(f_{0})$ and $b=\lc(f_{0})$. Assume that $ab\in U(R)$, then by elementary transformations of the form $$f_1-X^{\ldeg(f_1)-\ldeg(f_0)}b^{-1}\lc(f_1)f_0$$ we obtain that $f\sim_{E}(f_0,h_1,f_2,\dots,f_n)$, where $\ldeg(h_1)>\ldeg(f_0)$. By elementary transformations of the form $$f_1-X^{\hdeg(f_1)-\hdeg(f_0)}a^{-1}\hc(f_1)f_0$$ we obtain that $f\sim_{E}(f_0,g_1,f_2,\dots,f_n)$, where $\ldeg(g_1)\geq\ldeg(f_0)$ and $\hdeg(g_1)<\hdeg(f_0)$. So we may assume that $\deg f_{0}\leq \deg f_{1}$ and $ab\notin U(R)$. Assume that $a\notin U(R)$. We have $Ra=Re$ for some idempotent $e$. Let $c=\hc(f_1)$. Since $e\in Ra$, we may assume that $c\ne0$ and that $c\in R(1-e)$. Note that\begin{center} $(1-e)f=(f_{0}(1-e),\dots,f_{n}(1-e))\in \Um_{n+1}(R(1-e)[X,X^{-1}])$ and $ef=(f_{0}e,\dots,f_{n}e)\in \Um_{n+1}(Re[X,X^{-1}])$.\end{center}
By the inductive assumption, there are matrices $$A\in \E_{n+1}(R(1-e)[X,X^{-1}]),~B\in \E_{n+1}(Re[X,X^{-1}])$$ so that $(1-e)fA=(1-e,0,\dots,0)$ and $efB=(e,0,\dots,0).$ Let $$A=\prod_{s=1}^{k}\E_{ij}(h_{s}),~B=\prod_{s=1}^{t}\E_{ij}(g_{s})$$ where $$\E_{ij}(h_{s})=(1-e)I_{n+1}+h_{s}e_{ij},~\E_{ij}(g_{s})=eI_{n+1}+g_{s}e_{ij}$$ and $i\ne j\in \{1,\dots,n+1\},~h_{s}\in R(1-e)[X,X^{-1}],~g_{s}\in Re[X,X^{-1}]$. Let $$A'=\prod_{s=1}^{k}(I_{n+1}+h_{s}e_{ij}),~B'=\prod_{s=1}^{t}(I_{n+1}+g_{s}e_{ij}).$$
Clearly, $(1-e)A'=A$, $eB'=B$ and $A',B'\in \E_{n+1}(R[X,X^{-1}])$. Let $C=A'B'$, then $C\in \E_{n+1}(R[X,X^{-1}])$ and $$(1-e)C=(1-e)A'(1-e)B'=A(1-e)I_{n+1}=(1-e)A'=A.$$ Similarly, we have $eC=B$. Let $fC=(g_{0},\dots,g_{n})=g$. Thus \begin{center}$g_{0}(1-e)=1-e$ and $g_{1}e=e$.\end{center} So \begin{center}$f\sim_{\E_{n+1}(R[X,X^{-1}])}(g_0,\dots,g_n)\sim_{\E_{n+1}(R[X,X^{-1}])}(g_{0}+e,\dots,g_n)=(1+g_{0}e,\dots,g_{n})\sim_{\E_{n+1}(R[X,X^{-1}])}(1+g_{0}e,-g_{0}e,\dots,g_{n})\sim_{\E_{n+1}(R[X,X^{-1}])}e_{1}.$
\end{center}
Similarly, if $b\notin U(R)$, then $f\sim_{\E}e_1$.
\end{proof}

\begin{prop}\label{50}
If $R$ is a zero-dimensional ring, then
$$\SL_{n}(R[X,X^{-1}])=\E_{n}(R[X,X^{-1}])$$
for all $n\geq 2$.
\end{prop}
\begin{proof}
Clearly, $\E_{n}(R[X,X^{-1}])\subseteq \SL_{n}(R[X,X^{-1}])$. Let $M\in \SL_{n}(R[X,X^{-1}])$. By (\ref{48}), we can perform suitable elementary transformations to bring $M$ to $M_1$ with first row $(1,0,\dots,0)$. Now a sequence of row transformations bring $M_1$ to
$$M_{2}=\left(\begin{array}{cc}
1 & 0\\
0 & M'\end{array}\right)$$ where $M'\in \SL_{n-1}(R[X,X^{-1}])$. The proof now proceeds by induction on $n$.
\end{proof}

\begin{lem}\label{13}
Let $(f_0,\dots,f_n)\in \Um_{n+1}(R[X,X^{-1}])$, where $n\geq 2$. Assume that $\hc(f_0)$ is invertible modulo $f_0$. Then $$f\sim_{\E}(f_0,g_1,\dots,g_n)$$ where $\hdeg(g_i)<\hdeg(f_0),\ldeg(g_i)\geq\ldeg(f_0)$, for all $1\leq i\leq n$.
\end{lem}

\begin{proof}
By (\ref{2}), $f\sim_{\E}(f_0,X^{2k}f_1,\dots,X^{2k}f_n)$ for all $k\in \mathbb{Z}$. So we may assume that $\ldeg(f_i)>\ldeg(f_0)$. Let $a=\hc(f_0)$. By (\ref{2}) we have $$f\sim_{\E}(f_0,a^{2}f_1,\dots,a^{2}f_n).$$ Using elementary transformations of the form $$a^{2}f_i-aX^{\hdeg(f_i)-\hdeg(f_0)}\hc(f_i)f_0$$ we lower the degrees of $f_i$, for all $1\leq i\leq n$, and obtain the required row.
\end{proof}

\begin{lem}\label{10}
Let $R$ be a ring of dimension $d>0$ and $$f=(r,f_{1},\dots,f_{n})\in \Um_{n+1}(R[X,X^{-1}])$$ where $r\in R,~n\geq d+1$. Assume that for every ring $T$ of dimension $<~d$ and $n\geq\dim(T)+1$, the group $\E_{n+1}(T[X,X^{-1}])$ acts transitively on $\Um_{n+1}(T[X,X^{-1}])$. Then $f\sim_{\E(R[X,X^{-1}])}e_{1}$.
\end{lem}
\begin{proof}
Since $\dim(R/\mathcal{I}(r))<\dim(R)$ so over $R/\mathcal{I}(r)$, we can complete $(f_{1},\dots,f_{n})$ to a matrix in $\E_{n}(R/\mathcal{I}(r)[X,X^{-1}])$. If we lift this matrix, we obtain that \begin{center}
$(r,f_{1},\dots,f_{n})\sim_{\E_{n+1}(R[X,X^{-1}])}(r,1+rw_{1}+h_{1},\dots,rw_{n}+h_{n})\sim_{\E_{n+1}(R[X,X^{-1}])}(r,1+h_{1},\dots,h_{n})$ \end{center}
where $h_{i},~w_{i}\in R[X,X^{-1}]$ and $rh_{i}=0$ for all $1\leq i\leq n$. Then $$f\sim_{\E_{n+1}(R[X,X^{-1}])}(r-r(1+h_{1}),1+h_{1},\dots,h_{n})\sim_{\E_{n+1}(R[X,X^{-1}])}e_{1}.$$
\end{proof}

\begin{lem}\label{11}
Let $R$ be a ring of dimension $d>0$ and $$f=(f_{0},\dots,f_{n})\in \Um_{n+1}(R[X,X^{-1}])$$ such that $n\geq d+1$, $f_{0}=ag$ and $a^{t}=\hc(f_{0})$, where $a\in R\setminus U(R),~0\neq t\in \mathbb{N}$. Assume that for every ring $T$ of dimension $<~d$ and $n\geq\dim(T)+1$, the group $\E_{n+1}(T[X,X^{-1}])$ acts transitively on $\Um_{n+1}(T[X,X^{-1}])$. Then $f\sim_{\E(R[X,X^{-1}])}e_{1}$.
\end{lem}

\begin{proof}
We prove by induction on the number $M$ of non-zero coefficients of the polynomial $f_{0}$, that $f\sim_{\E}e_{1}$. If $M=1$, so $f_{0}=rX^{m}$ where $r\in R,m\in\mathbb{Z}$. By (\ref{2}), $f\sim_{\E}(r,X^{-m}f_{1},f_{2},\dots,f_{n})$. So by (\ref{10}), we obtain that $f\sim_{\E}e_{1}$.
Assume now that $M>1$. Let $S$ be the multiplicative subset of $R$ generated by $a,b$, where $b=\lc(g)$, i.e., $S=\setf{a^{k_{1}}b^{k_{2}}}{k_{1},k_{2}\in \mathbb{N}}$. By the inductive step, with respect to the ring $R/abR$, we obtain from $f$ a row $\equiv (1,0,\dots,0)\bmod{abR[X,X^{-1}]}$, also we can perform such transformation so that at every stage the row contains a doubly unitary polynomial in $R_{S}[X,X^{-1}]$, indeed, if we have to perform, e.g., the elementary transformation $$(g_{0},\dots,g_{n})\rightarrow (g_{0},g_{1}+hg_{0},\dots,g_{n})$$
and $g_{1}$ is a doubly unitary polynomial in $R_{S}[X,X^{-1}]$, then we replace this elementary transformation by the two transformations:
\begin{center}
$(g_{0},\dots,g_{n})\rightarrow (g_{0}+abX^{m}g_{1}+abX^{k}g_{1},g_{1},\dots,g_{n})\rightarrow (g_{0}+abX^{m}g_{1}+abX^{k}g_{1},g_{1}+h(g_{0}+abX^{m}g_{1}+abX^{-k}g_{1}),\dots,g_{n})$
\end{center}
where $m>\hdeg(g_{0}),~k<\ldeg(g_{0})$. So we may assume that $$(f_{0},\dots,f_{n})\equiv (1,0,\dots,0)\bmod{abR[X,X^{-1}]}$$ and $f_{0}$ is a doubly unitary polynomial in $R_{S}[X,X^{-1}]$. By (\ref{13}), we may assume that $\hdeg(f_{i})<\hdeg(f_{0}),~\ldeg(f_{i})\geq \ldeg(f_{0})$.

We prove that $f$ can be transformed by elementary transformation into a row with one constant entry. We use an argument similar to that in the proof of [\cite{E}, Theorem 5].

Assume that the number of the coefficients of $f_{2},\dots,f_{n}$ is $\geq 2(n-1)$. Since $d>0$, we obtain that $2(n-1)\geq d+1$.
Let $a_{1},\dots,a_{t}$ be the coefficients of $f_{2},\dots,f_{n}$ and $J=\frac{a_1}{1}R_{S}+\dots+\frac{a_t}{1}R_S$. Let $I=R_{S}[X,X^{-1}]f_{0}+R_{S}[X,X^{-1}]f_{1}$. Since $I+J[X,X^{-1}]=R_{S}[X,X^{-1}]$ and $f_{0}$ is a doubly unitary in $R_{S}[X,X^{-1}]$, by (\ref{42}), we obtain that $(I\cap R_{S})+J=R_{S}$. So $(\frac{f_{0}h_{0}+f_{1}h_{1}}{s})+\frac{r_{1}}{s_{1}}\frac{a_{1}}{1}+\dots+\frac{r_{t}}{s_{t}}\frac{a_{t}}{1}=\frac{1}{1}$ ,where $h_{0},h_{1}\in R[X,X^{-1}]$ and $r_{i}\in R,~s,s_{i}\in S$ for all $1\leq i\leq t$. This means that $(\frac{f_{0}h_{0}+f_{1}h_{1}}{1},\frac{a_{1}}{1},\dots,\frac{a_{t}}{1})\in \Um_{t+1}(R_{S})$. By (\ref{41}), there exist $s\in S$ and $b_{1},\dots,b_{t}\in R$, such that $$s\in (a_{1}+b_{1}(f_{0}h_{0}+f_{1}h_{1}))R+\dots+(a_{t}+b_{t}(f_{0}h_{0}+f_{1}h_{1}))R.$$ Using elementary transformations, we may assume that $J=R_{S}$. By (\ref{40}), the ideal $\langle f_{0},f_{2},\dots,f_{n}\rangle$ contains a polynomial $h$ such that $a^{k_{1}}b^{k_{2}}=\hc(h)$ and $\hdeg(h)=\hdeg(f_{0})-1,\ldeg(h)\geq \ldeg(f_{0})$ where $k_{1},k_{2}\in \mathbb{N}$. Let $r=\hc(f_{1})$, So
\begin{center}
$f\sim_{\E}(f_{0},a^{2k_{1}}b^{2k_{2}}f_{1},f_{2},\dots,f_{n})\sim_{\E}(f_{0},a^{2k_{1}}b^{2k_{2}}f_{1}+(1-a^{k_{1}}b^{k_{2}}r)h,f_{2},\dots,f_{n}).$
\end{center}
Then we may assume that $a^{k_{1}}b^{k_{2}}=\hc(f_{1})$. By the proof of Lemma (\ref{13}), we can decrease the $\hdeg(f_{i})$ for all $2\leq i\leq n$.

Repeating the argument above, we obtain that $$f\sim_{\E}(rX^{m},g_{1},\dots,g_{n})\sim_{\E}(r,g_{1}X^{-m},g_{2},\dots,g_{n})$$ where $r\in R,m\in\mathbb{Z}$, $g_{1},\dots,g_{n}\in R[X,X^{-1}]$. By (\ref{10}), $f\sim_{\E}e_{1}$.
\end{proof}

\begin{lem}\label{12}
Let $R$ be a ring of dimension $d>0$ and $$f=(f_{0},\dots,f_{n})\in \Um_{n+1}(R[X,X^{-1}])$$ such that $n\geq d+1$, $f_{0}=cg$ and $c^{t}=\lc(f_{0})$, where $c\in R\setminus U(R),~0\neq t\in \mathbb{N}$. Assume that for every ring $T$ of dimension $<~d$ and $n\geq\dim(T)+1$, the group $\E_{n+1}(T[X,X^{-1}])$ acts transitively on $\Um_{n+1}(T[X,X^{-1}])$. Then $f\sim_{\E(R[X,X^{-1}])}e_{1}$.
\end{lem}
\begin{proof}
By making the change of variable: $X\rightarrow X^{-1}$ and Proposition (\ref{11}), we obtain that $f\sim_{\E(R[X,X^{-1}])}e_{1}$.
\end{proof}

\begin{thm}\label{51}
Let $R$ be a ring of dimension $d$ and $n\geq d+1$, then $\E_{n+1}(R[X,X^{-1}])$ acts transitively on $\Um_{n+1}(R[X,X^{-1}])$.
\end{thm}
\begin{proof}
Let $f=(f_0,\dots,f_n)\in \Um_{n+1}(R[X,X^{-1})$. We prove the theorem by induction on $d$, we may assume that $R$ is reduced ring. If $d=0$, by (\ref{48}), we are done.
Assume that the theorem is true for the dimensions $0,1,\dots,d-1$, where $d>0$.
We prove by induction on the number $N$ of nonzero coefficients of the polynomials $f_{0},\dots,f_{n}$, that $f\sim_{\E}e_{1}$ if $\dim R=d$. Starting with $N=1$. Let $N>1$. Let $a=\hc(f_{0})$ and $c=\lc(f_{0})$, if $ac\in U(R)$ then by (\ref{43}), we are done. Otherwise, assume that $a\notin U(R)$, by the inductive step, with respect to the ring $R/aR$, we obtain from $f$ a row $\equiv (1,0,\dots,0)\bmod{aR[X,X^{-1}]}$ using elementary transformations. We can perform such transformations so that at every stage the row contains a polynomial $g\in R[X,X^{-1}]$ such that $\hc(g)=a^{t}$, where $t\in \mathbb{N}$. Indeed, if we have to perform, e.g., the elementary transformation $$(g_{0},\dots,g_{n})\rightarrow (g_{0},g_{1}+hg_{0},\dots,g_{n})$$
and $\hc(g_{1})\in U(R_{a})$, then we replace this elementary transformation by the two transformations:
\begin{center}
$(g_{0},\dots,g_{n})\rightarrow (g_{0}+aX^{m}g_{1},g_{1},\dots,g_{n})\rightarrow (g_{0}+aX^{m}g_{1},g_{1}+h(g_{0}+aX^{m}g_{1}),\dots,g_{n})$
\end{center}
where $m>\hdeg(g_{0})$.

So we have $f_{0}=ag$, and $a^{t}=\hc(f_{0})$, where $0\neq t\in \mathbb{N}$.
By (\ref{11}), $f\sim_{\E}e_{1}$. Similarly, if $c\notin U(R)$, by (\ref{12}) we obtain that $f\sim_{\E}e_{1}$.
\end{proof}

\begin{cor}\label{52}
For any ring $R$ with Krull dimension $\leq d$, all finitely generated stably free modules over $R[X,X^{-1}]$ of rank $>d$ are free.
\end{cor}

The following conjecture is the analogue of Conjecture 8 of \cite{G} in the Laurent case:

\begin{con}
For any ring $R$ with Krull dimension $\leq d$, all finitely generated stably free modules over $R[X_{1}^{\pm1},\dots,X_{k}^{\pm1},X_{k+1},\dots,X_{n}]$ of rank $>d$ are free.
\end{con}
\textbf{Acknowledgments.}
I would like to thank Professor Moshe Roitman, my M.Sc. thesis advisor, for his interest in this project.

\bibliographystyle{amsplain}

\end{document}